\documentclass[11pt]{amsart}
\usepackage{amssymb}
\usepackage{amsmath}
\usepackage{amsfonts}
\usepackage[usenames]{color}
\usepackage{graphicx}

\makeatletter
 
 \@addtoreset{equation}{section}
\makeatother

\newtheorem{theorem}{Theorem}[section]

\theoremstyle{definition}
\newtheorem{definition}[theorem]{Definition}

\newtheorem{proposition}[theorem]{Proposition}
\newtheorem{conjecture}[theorem]{Conjecture}

\theoremstyle{remark}
\newtheorem{remark}[theorem]{Remark}
\newtheorem{question}[theorem]{Question}

\numberwithin{equation}{section}

\newcommand{\R}{\mathbb{R}}
\newcommand{\Z}{\mathbb{Z}}

\newcommand{\U}{\mathcal{U}}

\begin{document}

\title[The $TB$ of complete and complete bipartite graphs]{The total Thurston-Bennequin number of complete and complete bipartite Legendrian graphs}

\author[Danielle O'Donnol]{Danielle O'Donnol$^{\dag}$}
\address{Department of Mathematics, Oklahoma State University, Stillwater, OK 74078}
\email{odonnol@okstate.edu}

\author{Elena Pavelescu}
\address{Department of Mathematics, Oklahoma State University, Stillwater, OK 74078}
\email{elena.pavelescu@okstate.edu}

\thanks{$^{\dag}$ This work was partially supported by the National Science Foundation grant DMS-1406481.} 

%    General info
\subjclass[2010]{Primary 57M25, 57M50;  Secondary 05C10}

\date{\today}

\keywords{Legendrian graph, Thurston-Bennequin number, complete graph, complete bipartite graph}

\begin{abstract}
We study the Thurston-Bennequin number of complete and complete bipartite Legendrian graphs. 
We define a new invariant called the total Thurston-Bennequin number of the graph.
We show that this invariant is determined by the Thurston-Bennequin numbers of $3-$cycles for complete graphs and by the Thurston-Bennequin number of $4-$cycles for complete bipartite graphs. 
We discuss the consequences of these results for $K_4$, $K_5$ and $K_{3,3}$.
 \end{abstract}

\maketitle

%------------------ Introduction  --------------------------------------

\section{Introduction}\label{intro}
Motivated by the appearance of Legendrian graphs in important results, the authors began studying them \cite{ODPa1}. 
Two nice examples of such results are  Giroux's proof of existence of open book decompositions compatible with a given contact structure \cite{G}, and 
Eliashberg and Fraser's proof of the Legendrian simplicity of the unknot \cite{EF}. 
We anticipate that with a better understanding of Legendrian graphs, they will become an even more robust tool.

In \cite{ODPa1}, the authors extended the classical invariants Thurston-Bennequin number, $tb$, and rotation number, $rot$, from Legendrian knots to Legendrian graphs. 
Throughout this paper we work in $\mathbb{R}^3$ with the standard contact structure, $\xi_{std}$.
The Thurston-Bennequin number measures the number of times the contact planes twist around the knot as the knot is traversed once, and can be computed from the front projection as $tb(K)=w(K)-\frac{1}{2}c(K)$, where $w(K)$ is the writhe and $c(K)$ is the number of cusps.  
For a Legendrian graph, $G,$ with a fixed order on its cycles, the {\it Thurston-Bennequin number}, $tb(G)$, is the ordered list of the Thurston-Bennequin numbers of its cycles.
Similarly, the {\it rotation number}, $rot(G)$, is the ordered list of the rotation numbers of its cycles.
It is known \cite{El} that if $K$ is a Legendrian knot in  $(\R^3, \xi_{std})$ and $\Sigma$ is a Seifert surface for $K$, then $$tb(K) + |rot (K)| \leq - \chi(\Sigma).$$
This inequality puts an upper bound on the $tb(K)$.
\noindent In particular,  if $K$ is the unknot, then $tb(K) \le -1$. 
There is a unique Legendrian unknot with maximal $tb(K) = -1,$ it is called the {\it trivial unknot}.   
For a topological knot type $\mathcal{K}$, we denote by $tb_{max}(\mathcal{K})$ the maximum $tb$ among all Legendrian knots of topological type $\mathcal{K}$.

In this paper, we introduce a new invariant of Legendrian graphs, called the total Thurston-Bennequin number. 
The {\it total Thurston-Bennequin number}, $TB(G)$, is the sum of $tb$s over all cycles of $G$.  
We derive a simplified diagrammatic means of computing $TB(G)$ for complete graphs and complete bipartite graphs, and show that it depends only on the $tbs$ of the smallest cycles.  
In particular cases, we show that the Thurston-Bennequin number $tb$ of a graph is determined by the total Thurston-Bennequin number $TB$.

The main theorems relate the total Thurston-Bennequin number of complete and complete bipartite graphs with the sum of the $tbs$ of their smallest cycles.  
For a Legendrian embedding of the complete graph on $n$ vertices, $K_n$, in $(\R, \xi_{std}),$ we show:
\[
\begin{array}{ccl}
TB(K_n) & = & TB_3(K_n) \cdot \displaystyle\sum_{r=3}^{n}\frac{(n-3)!}{(n-r)!} \\\\
  & = & \left[(n-2)w_e(K_n)+w_{ae}(K_n) -\frac{1}{2}[(n-2)c_e(K_n)+c_v(K_n) ] \right] \cdot \displaystyle \sum_{r=3}^{n}\frac{(n-3)!}{(n-r)!},
\end{array}
\]
where $TB_3(K_n)$ is the sum of the $tbs$ over all $3-$cycles of $K_n,$ and $w_{\ast}$ and $c_{\ast}$ indicate different writhes and cusp counts described in Section \ref{GTBC}.  
For a Legendrian embedding of the complete bipartite graph, $K_{n,m}$ with $m\leq n$ in $(\R, \xi_{std}),$ we show:
$$TB(K_{n,m})= TB_4(K_{n,m})\cdot \sum_{r=2}^{m} \frac{(m-2)!(n-2)!}{(m-r)!(n-r)!}, $$
where $TB_4(K_{n,m})$ is the sum of the $tbs$ over all $4-$cycles of $K_{n,m},$ and \\\\ 
$TB_4(K_{n,m})=$
\[
\begin{array}{cl}
  & (n-1)(m-1)w_e(K_{n,m})+(n-1)w_{ae[P]}(K_{n,m})+(m-1)w_{ae[Q]}(K_{n,m})+w_{ne}(K_{n,m}) \\\\
  & -\frac{1}{2}[(n-1)(m-1)c_e(K_{n,m})+(n-1)c_{v[P]}(K_{n,m})+(m-1)c_{v[Q]}(K_{n,m})].
\end{array}
\]
Here $w_{\ast}$ and $c_{\ast}$ indicate different writhes and cusp counts described in Section \ref{GTBB}.

Let a \textit{minimal embedding} of $G$ be one where all minimal length cycles are unknots with $tb=-1$.  
An  \textit{unknotted minimal embedding} of $G$ is a minimal embedding where all the cycles are unknots.
We give some examples of minimal embeddings of $K_4$, $K_5$ and $K_{3,3}$.  
For unknotted minimal embeddings of $K_4$ and $K_{3,3}$, using our understanding of $TB$ and the graphs structure, we show they have unique $tb$ and $|rot|$ up to relabelling of the cycles.
We give a lower bound for the $tb$ of an unknotted $r-$cycle in an embedding of $K_n$ with all $3-$cycles trivial unknots.\\

{\bf Acknowledgements.}  We would like to thank Youngjin Bae, Byung Hee An and Gabriel C. Drummond-Cole for useful conversation, and Tim Cochran and John Etnyre for their interest and support.

%--------------The graph Thurston-Bennequin number ----------------

\section{The graph Thurston-Bennequin number for complete graphs}\label{GTBC}
In this section we introduce much of the notation and definitions that are used throughout this article.  
This section is focused on the total Thurston-Bennequin number of a complete graph.  
Here we prove the main theorem for the $TB$ of complete graphs and look at the consequences for embeddings of $K_4$ and $K_5$.  

We first introduce notation and definitions. 
\begin{definition}For a Legendrian graph $G$, we define the \textit{total Thurston-Bennequin number of $G$}, $TB(G)$, as the sum of Thurston-Bennequin numbers over all cycles of $G$.
For a Legendrian graph $G$, we define $TB_r(G)$ as the sum of Thurston-Bennequin numbers  over all $r-$cycles of $G$.
\end{definition}

\noindent 
A  \textit{minimal embedding} of $G$ is a Legendrian embedding where all minimal length cycles are trivial unknots.
We denote by $G[\mathcal{K}]$ an embedding of the graph $G$ for which all cycles have the knot type $\mathcal{K}$.
We denote the unknot with $\mathcal{U}$.  
An embedding $G[\U]$ with all minimal length cycles trivial unknots is called an \textit{unknotted minimal embedding}.  
The set of all $r-$cycles of a graph $G$ is denoted by $\Gamma_r(G)$, or simply $\Gamma_r$ when $G$ is understood.
For a cycle $\gamma$ of $G$, we denote by $w(\gamma)$ and $c(\gamma)$ the \textit{writhe of $\gamma$} (signed sum of crossings of $\gamma$) and the \textit{number of cusps of $\gamma$}, respectively. 
For edges $e$, $f$ and vertex $v$, we denote by $w(e)$ the writhe of edge $e$ with itself, by $w(e,f)$ the signed sum of crossings between edges $e$ and $f$, by $c(f)$ the number of cusps along edge $f$, and by $c(v)$ the number of cusps at vertex $v$, looking at each pair of edges going through $v$. 
For a front diagram of a Legendrian graph $G$, we define 
\begin{itemize}
\item the \textit{edge writhe of  $G$} as the sum of writhes over all edges of $G$, $w_e(G)=\sum_{f\in E(G)} w(f)$.
\item the \textit{adjacent edge writhe of $G$} as the sum of writhes over all pairs of adjacent edges of $G$, $w_{ae}(G)=\sum_{e, f\in E(G) \textrm{adj}} w(e,f)$.
\item the \textit{non-adjacent edge writhe  of $G$} as the sum of writhes over all pairs of non-adjacent edges of $G$, $w_{ne}(G)=\sum_{e, f\in E(G) \textrm{non-adj}} w(e,f)$.
\item the \textit{edge cusps of $G$} as the number of cusps along all edges of $G$, $c_{e}(G)=\sum_{f\in E(G)}  c(f)$.
\item the \textit{vertex cusps of $G$} as the count of cusps at all vertices of $G$,  $c_{v}(G)= \sum_{v\in V(G)} c(v) $.

\end{itemize}

\begin{theorem} Let $K_n$ be a Legendrian embedding of the complete graph on $n$ vertices in $(\R^3, \xi_{std})$.  
Then $TB_r(K_n)$ is a multiple of $TB_3(K_n)$. 
As a consequence, $TB(K_n)$ is a multiple of $TB_3(K_n)$. 
The quantities $TB_i(K_n)$, $3\le i\le n$,  can be computed from writhe and cusp counts of vertices and edges in the front projection rather than summing the tbs of the cycles.
\label{TBKn}
\end{theorem}

\begin{proof}
When computing the sum of writhes over the $3-$cycles or the $r-$cycles of $K_n$, we consider crossings of an edge with itself, crossings between adjacent edges and crossings between non-adjacent edges.
\begin{enumerate}
 \item  Each edge of $K_n$ appears in $n-2$ of the $3-$cycles and in ${n-2 \choose r-2} (r-2)! = \frac{(n-2)!}{(n-r)!}$ $r-$cycles. 
\item Each pair of adjacent edges of $K_n$ appears in one $3-$cycle and in $\frac{(n-3)!}{(n-r)!}$ $r-$cycles. 
\item Non-adjacent edges do not contribute to the total writhe of the graph. 
The cycles containing both edges $ab$ and $cd$ can be paired as  $...ab...cd...$ and 
 $ ...ab...dc... $. The signed intersection of the two edges in one of the cycles is the negative of their signed intersection in the other cycle. 
 This means that the crossings between the two edges do not contribute to the sum of the writhes. 
    \end{enumerate}

  Items (1)--(3) above give $$\sum_{\gamma \in \Gamma_r} w(\gamma) = \frac{(n-3)!}{(n-r)!} \sum_{\gamma \in \Gamma_3} w(\gamma) =  \frac{(n-3)!}{(n-r)!} [(n-2)w_e(K_n)+w_{ae}(K_n) ].
 $$
Cusps occur either at a vertex, that is, at each pair of adjacent edges, or along one edge. Using (1) and (2) above, we have 
$$\sum_{\gamma \in \Gamma_r} c(\gamma) = \frac{(n-3)!}{(n-r)!} \sum_{\gamma \in \Gamma_3} c(\gamma)=  \frac{(n-3)!}{(n-r)!} [(n-2)c_e(K_n)+c_v(K_n) ]. $$
For every cycle $\gamma$, $tb(\gamma) = w(\gamma) - \frac{1}{2}c(\gamma)$. Then the two identities above give 
\[
\begin{array}{ccl}
TB_r(K_n)  & = & \frac{(n-3)!}{(n-r)!} TB_3(K_n)  \\\\
  & = & \frac{(n-3)!}{(n-r)!}\left[(n-2)w_e(K_n)+w_{ae}(K_n) -\frac{1}{2}[(n-2)c_e(K_n)+c_v(K_n)] \right]  .
\end{array}
\]
Adding over $r$ gives
\[
\begin{array}{ccl}
TB(K_n) & = & TB_3(K_n) \cdot \displaystyle\sum_{r=3}^{n}\frac{(n-3)!}{(n-r)!} \\\\
  & = & \left[(n-2)w_e(K_n)+w_{ae}(K_n) -\frac{1}{2}[(n-2)c_e(K_n)+c_v(K_n) ] \right] \cdot \displaystyle\sum_{r=3}^{n}\frac{(n-3)!}{(n-r)!}. 
\end{array}
\]

\end{proof}

Theorem \ref{TBKn} has many consequences. 
We focus on minimal embeddings.
%We describe the front projection of  such an embedding here. 
For the front project of such an embedding, place all vertices of $K_n$ on the same horizontal line in the front projection.
Then place the edges adjacent to a given vertex as nested arcs, with no intersections between adjacent edges.
See the left image of Figure \ref{fig-Kn} for such an embedding of $K_6.$
All $3-$cycles in this embedding have a front projection like that of the unknot on the right in the picture and therefore are trivial unknots.

%---- Figure ---------------------K4examples----------

\begin{figure}[htpb!]
\begin{center}
\begin{picture}(500, 85)
\put(0,0){\includegraphics[width=6in]{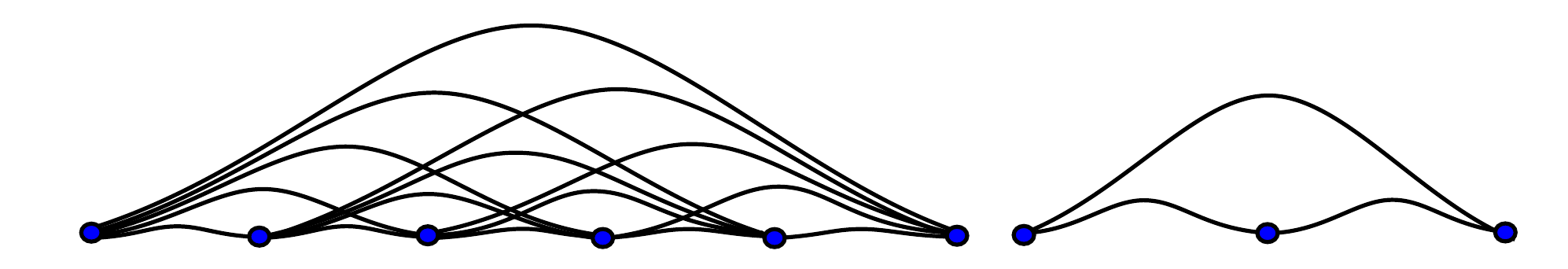}}
\end{picture}
\caption{\small A minimal embedding of $K_6$ and a $3-$cycle in this embedding. }
\label{fig-Kn}
\end{center}
\end{figure}

 \begin{subsection}{Remarks about $K_4$}
 \label{subsecK4}

 In \cite{ODPa1}, the authors showed there does not exist a Legendrian embedding of $K_4[\mathcal{U}]$ where all cycles are trivial unknots.
 Moreover, $TB(K_4[\mathcal{U}])$ is shown to be at most $-8$.
  Theorem  \ref{TBKn} gives a refinement of this result. For $n=4$, we have $$TB_4(K_4)= TB_3(K_4).$$
The graph $K_4$ has seven cycles, four cycles of length 3 and three  of length 4.
This means that for a Legendrian embedding of $K_4[\mathcal{U}]$ with all trivial $3-$cycles we have
$$TB_4(K_4)= TB_3(K_4) = -4. $$
For such an embedding, exactly one of its three $4-$cycles has $tb=-2$ and the other two have $tb=-1$.
This is the only way to have $TB(K_4[\mathcal{U}])=-8$. 
If all the $4-$cycles have $tb=-1$, then $TB(K_4[\mathcal{U}]) =-6$, which is a contradiction.

The $4-$cycles of $K_4$ are equivalent under graph automorphism, so up to a relabelling of the cycles there is a unique $tb(K_4)$ for an unknotted  minimal embedding of $K_4$.  
Since unknots with $tb=-2$ and $tb=-1$ are unique, there is only one possible $tb(K_4)$ and $|rot(K_4)|$ for an unknotted  minimal embedding of $K_4$.  
In Figure \ref{K4examples}, we show two diagrams of an unknotted minimal embedding of $K_4$.  
(See Figures \ref{K4flat-flat} and \ref{K4other-flat} for the equivalence.)  
While there is a unique $tb$ and $|rot|$, this does not mean there is only one unknotted minimal embedding of $K_4$.

%---- Figure ---------------------K4examples----------

\begin{figure}[htpb!]
\begin{center}
\begin{picture}(330, 75)
\put(0,0){\includegraphics[width=5in]{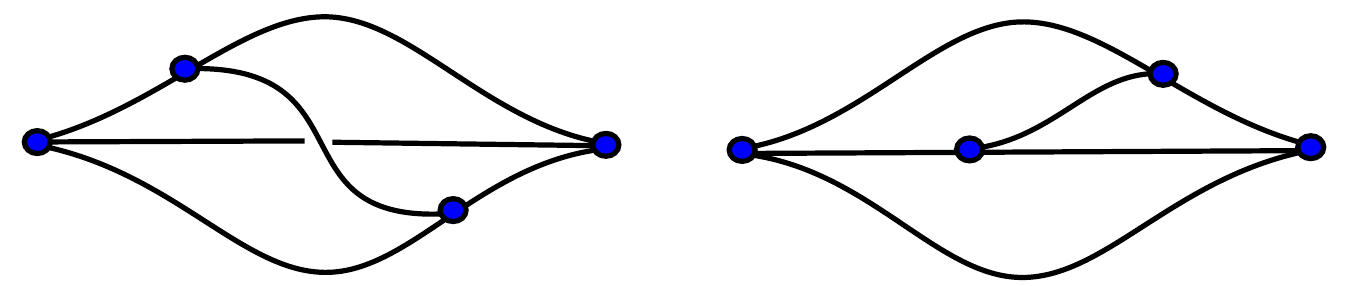}}
\put(173, 33){\Large $\simeq$}
\end{picture}
\caption{\small Two diagrams of an unknotted minimal embedding of $K_4$.}
\label{K4examples}
\end{center}
\end{figure}

\begin{remark}
In more generality, for $K_n$ with all $3-$cycles trivial unknots and $4-$cycles unknots, one third of $4-$cycles have $tb=-2$ and two thirds of $4-$cycles have $tb=-1$. 
This is because for every $K_4$ subgraph of $K_n$,  exactly one of three $4-$cycles has $tb=-2$ and the other two have $tb=-1$.
\end{remark}

\begin{remark}
The graph $K_4$ is adaptable \cite{Ya}, that is, given any set of seven knot types, there exists an embedding of $K_4$ with its seven cycles realizing the seven knot types. 
All topological embeddings of a graph have a Legendrian realization \cite{ODPa1}.
Since $TB_4(K_4)= TB_3(K_4)$, if all cycles of $K_4$ are of the same knot type, $L$, with $tb_{max}(L) >0$, then at least one of the $3-$cycles has non-maximal $tb$.
On the other hand, if all cycles of $K_4$ are of the same knot type, $L$,  with $tb_{max} (L)< 0$, then at least one of the $4-$cycles has non-maximal $tb$.
\end{remark}

\end{subsection}

\subsection{Remarks about $K_5$} We look at unknotted minimal Legendrian embeddings of $K_5$.
Since $K_5$ has ten $3-$cycles, Theorem \ref{TBKn} for  $n=5$ and $r=4$ says $$TB_4(K_5) =2TB_3(K_5)=-20.$$
Each of five $K_4$ subgraphs of $K_5$ contains exactly one $4-$cycle with $tb=-2$. 
This means that among the fifteen $4-$cycles of $K_5$, ten have $tb=-1$ and five have $tb=-2$.

For $n=5$ and $r=5$, Theorem \ref{TBKn} says
$$TB_5(K_5) =2TB_3(K_5)=-20.$$
In Proposition \ref{bound}, we show that the minimum $tb$ for an unknotted $5-$cycle is $-4$.
There are ten possible ways to write $-20$ as a sum of twelve integers in the set $\{-4, -3, -2, -1\}$.
These ten sequences are candidates for the $tb$s of the $5-$cycles of $K_5$:
$$s_1=(-1, -1, -1, -1, -2, -2, -2, -2, -2, -2, -2, -2) $$ 
$$s_2=(-1, -1, -1, -1, -1, -2, -2, -2, -2, -2, -2, -3) $$ 
$$s_3= (-1, -1, -1, -1, -1, -1, -2, -2, -2, -2, -3, -3)$$
$$s_4= (-1, -1, -1, -1, -1, -1, -1, -2, -2, -3, -3, -3)$$
$$s_5= (-1, -1, -1, -1, -1, -1, -1, -1, -3, -3, -3, -3)$$
$$s_6=(-1, -1, -1, -1, -1, -1, -2, -2, -2, -2, -2, -4) $$ 
$$s_7=(-1, -1, -1, -1, -1, -1, -1, -2, -2, -2, -3, -4) $$ 
$$s_8= (-1, -1, -1, -1, -1, -1, -1, -1, -2, -3, -3, -4)$$
$$s_9= (-1, -1, -1, -1, -1, -1, -1, -1, -2, -2, -4, -4)$$
$$s_{10}= (-1, -1, -1, -1, -1, -1, -1, -1, -1, -3, -4, -4)$$

Unlike unknotted minimal embeddings of $K_4$, an unkotted minimal embedding of $K_5$ can have cycles with $tb<-2$.
So it can have cycles where the Bennequin inequality given by $tb(\mathcal{U})+|rot(\mathcal{U})|\le -1$ (here $\mathcal{U}$ denotes an unknotted cycle) is strict.
We give examples realizing sequences $s_2$ and $s_3$.
See Figure \ref{K5examples}.
The leftmost $K_5$ in Figure \ref{K5examples}, has one $5-$cycle with $tb=-3$. 
This cycle has rotation number $0$. 
The middle $K_5$ in Figure \ref{K5examples}, has one $5-$cycle with  $tb=-3$. 
This cycle has rotation number $\pm 2$, depending on the chosen orientation. 
This middle $K_5$ is the only embedding we have found where all its cycles are unknots $U$ such that $tb(U)+|rot(U)|=-1$.
Thus the Bennequin bound is also sharp for $K_5 (\mathcal{U})$.
The rightmost  $K_5$ in Figure \ref{K5examples}, has two $5-$cycles with  $tb=-3$ (the highlighted cycle and its symmetric about the middle vertical)  and both of these cycles have rotation number 0.

%------ Figure K5examples----------

\begin{figure}[htpb!]
\begin{center}
\begin{picture}(400, 60)
\put(0,0){\includegraphics[width=6in]{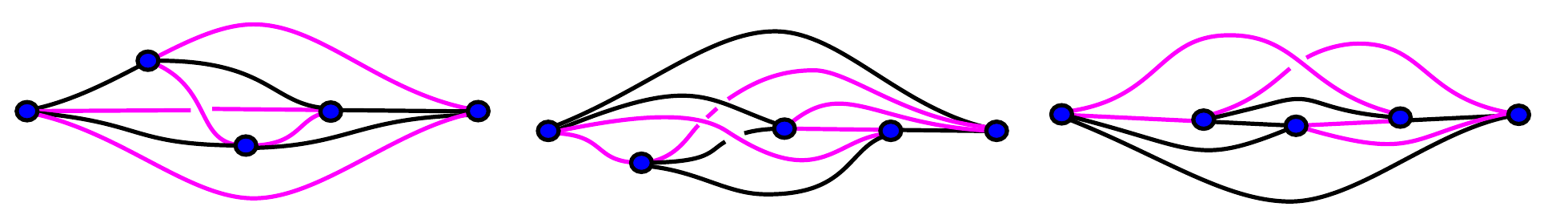}}
\end{picture}
\caption{\small Unknotted minimal embeddings of $K_5$ realizing $s_2$ (left and middle) and $s_3$ (right). The  highlighted cycles have $tb=-3$.}
\label{K5examples}
\end{center}
\end{figure}

\begin{proposition}
Let $\gamma_r$ be an unknotted $r-$cycle in $K_n$ with all $3-$cycles trivial unknots.
Then $$tb(\gamma_r)\ge c_r := \frac{2(s_r^2-1)}{3}-(r-2)s_r,$$ where $s_r=2^{\lfloor\log_2(r-2)\rfloor}$ and $\lfloor a \rfloor$ represents the largest integer not grater than $a$.
%If $c_r$, $r\ge 3$, is the minimal $tb$ of an unknotted $r-$cycle in $K_n$ with all $3-$cycles trivial unknots, then $$c_r\ge \frac{2(s_r^2-1)}{3}-(r-2)s_r,$$ where $s_r=2^{\lfloor\log_2(r-2)\rfloor}$ and $\lfloor a \rfloor$ represents the largest integer not grater than $a$.
\label{bound}
\end{proposition}
\begin{proof} By hypothesis, $tb(\gamma_3)=-1$. 
Also $c_3=-1=  \frac{2(s_3^2-1)}{3}-(3-2)s_3,$ since $s_3=1$.
Consider the $K_4$ subgraph determined by the four vertices in a $4-$cycle.
From our discussion in Subsection \ref{subsecK4}, we know that $tb(\gamma_4)\ge -2$.  
Note $c_4=-2= \frac{2(s_4^2-1)}{3}-2s_4$.  

Consider a $5-$cycle, $\gamma_5$.
Think of  $\gamma_5$ as obtained by  adding a vertex along an edge of a $4-$cycle of a $K_4$ subgraph.
The cycles containing this edge gain one unit in length.
We get a subdivision of $K_4$ where the identity 
$TB_3(K_4)= TB_4(K_4)$ gives  $$tb(\gamma_4) +tb(\gamma'_4)+tb(\gamma_3) + tb(\gamma'_3)= tb(\gamma''_4) +tb(\gamma_5)+tb(\gamma'_5). $$
Here $\gamma_i$, $\gamma'_i$ and $\gamma''_i$ are $i-$cycles.
We want to minimize $tb(\gamma_5)$. 
Since $c_3=-1$ and $c_4=-2$, the lefthand side of the equality is at least $-6$. 
The righthand side of the equality is at most $tb(\gamma_5)-2$, when $tb(\gamma''_4)=tb(\gamma'_5)=-1$.
We have $-6 \le tb(\gamma_5)-2$, and so $tb(\gamma_5) \ge -4$. 

We do the same for a $6-$cycle, $\gamma_6$.
Think of  $\gamma_6$ as obtained by adding two vertices  along edges of a $4-$cycle of $K_4$.
These two vertices can be added (1) on the same edge, (2) on two adjacent edges or (3) on two non-adjacent edges. 
In each case, we get a subdivision of $K_4$ where we use the identity $TB_3(K_4)= TB_4(K_4)$. 
On the righthand side of the equality we have the $6-$cycle $\gamma_6$ and two other cycles. 
Since we want to find a lower bound for $tb(\gamma_6)$, we assume that $tb=-1$ for the other two cycles. 
Then the righthand side of the identity is at most $tb(\gamma_6)-2$. 
For each of the three cases the identity $TB_3(K_4)= TB_4(K_4)$ gives:
\begin{enumerate}
\item $tb(\gamma_3) +tb(\gamma'_3)+tb(\gamma_5) + tb(\gamma'_5) \le tb(\gamma_6) -2 $\vspace{1mm}
\item $tb(\gamma_3) +tb(\gamma_4)+tb(\gamma'_4) + tb(\gamma_5) \le tb(\gamma_6) -2 $\vspace{1mm}
\item $tb(\gamma_4) +tb(\gamma'_4)+tb(\gamma''_4) + tb(\gamma'''_4) \le tb(\gamma_6) -2 $
\end{enumerate}
Here the $\gamma_i$s represent $i-$cycles.
Since we added two vertices to various edges of $K_4$, and each edge of $K_4$ appears in two $3-$cycles,  the total length of the four cycles on the lefthand side of the identity is $4\cdot 3 + 2\cdot 2 =16$. 
Using the lower bounds for the $tb$s of $3-$cycles ($c_3=-1$), $4-$cycles ($c_4=-2$) and $5-$cycles ($c_5=-4$) we get 
\begin{enumerate}
\item $-1-1-4-4 \le tb(\gamma_6) -2 \Rightarrow tb(\gamma_6)\ge -8$\vspace{1mm}
\item $-1-2-2-4 \le tb(\gamma_6) -2 \Rightarrow tb(\gamma_6)\ge -7$\vspace{1mm}
\item $-2-2-2-2 \le tb(\gamma_6) -2 \Rightarrow tb(\gamma_6)\ge -6$
\end{enumerate}
So $tb(\gamma_6) \ge -6$, and one can check that $c_6=-6$.

We proceed in the same way in the general case. 
Think of a $k-$cycle $\gamma_k$ as obtained by adding $k-4$ vertices to the edges of a $4-$cycle of $K_4$.
There are many possible choices, and in each case we get a subdivision of $K_4$ where we use the identity $TB_3(K_4)= TB_4(K_4)$.
On the righthand side of the equality we always have the $k-$cycle $\gamma_k$ and two other cycles which we assume have $tb=-1$.
Then the righthand side of the identity is at most $tb(\gamma_k)-2$. 
Since we added $k-4$ vertices to various edges of $K_4$, the total length of the four cycles on the lefthand side of the identity is $4\cdot 3 + 2(k-4) =2k+4$. 

The lefthand side of the identity can take on various forms 
$$S_{n_1, n_2, n_3, n_4}= tb(\gamma_{n_1}) +tb(\gamma_{n_2})+tb(\gamma_{n_3})+ tb(\gamma_{n_4}),$$ with  $n_1+n_2+n_3+n_4 =2k+4, n_i\ge 3$. 
As in the $r=6$ case above, each choice provides a lower bound for $tb(\gamma_k)$.  \\

\noindent{\bf For $k$ even:}  Let $k=2k_1$.  
One lower bound for $tb(\gamma_k)$ is given by $n_1 =n_2=n_3=n_4 =k_1+1$. 
 We obtain this configuration by placing $k_1-2$ vertices each, on two non-adjacent edges of the $4-$cycle.
 Then $tb(\gamma_k)\ge 2 +4 tb(\gamma_{k_1+1})$. 
 We use mathematical induction to prove that $tb(\gamma_k)\ge \frac{2(s_k^2-1)}{3}-(k-2)s_k.$
 Assume 
 $$ tb(\gamma_{k_1+1})\ge  \frac{2(s_{k_1+1}^2-1)}{3}-(k_1-1)s_{k_1+1}. $$
 Then 
 $$tb(\gamma_k))\ge 2 +4 tb(\gamma_{k_1+1})\ge 2+\frac{8}{3} (s_{k_1+1}^2-1)-4(k_1-1)s_{k_1+1}.$$
  We show this last quantity is  equal to $\frac{2}{3}(s_k^2-1)-(k-2)s_k.$
We have  
\[
\begin{array}{cl}
\frac{2}{3}(s_k^2-1)-(k-2)s_k & = \frac{2}{3} (2^{2\lfloor\log_2(2k_1-2)\rfloor}-1)-(2k_1-2)2^{\lfloor\log_2(2k_1-2)\rfloor}  \\\\
  & = \frac{2}{3} (2^{2+2\lfloor\log_2(k_1-1)\rfloor}-1)-2(k_1-1)2^{1+\lfloor\log_2(k_1-1)\rfloor} \\\\
  & =\frac{8}{3} 2^{2\lfloor\log_2(k_1-1)\rfloor}-\frac{2}{3}-4(k_1-1)2^{\lfloor\log_2(k_1-1)\rfloor} \\\\
    & =\frac{8}{3} s_{k_1+1}^2-\frac{2}{3}-4(k_1-1)s_{k_1+1} \\\\
  & = 2+\frac{8}{3} (s_{k_1+1}^2-1)-4(k_1-1)s_{k_1+1}.  
\end{array}
\]
Thus 
$$tb(\gamma_k))\ge  \frac{2}{3}(s_k^2-1)-(k-2)s_k.$$

\noindent{\bf For $k$ odd:}  Let $k=2k_1+1$.  
One lower bound $tb(\gamma_k)$ is given by  $n_1 =n_2 =k_1+1$ and $n_3=n_4 =k_1+2$. 
We obtain this configuration by placing vertices on a pair of non-adjacent edges of the $4-$cycle, $k_1-2$ vertices on one edges and $k_1-1$ vertices on the other edge.
Then $tb(\gamma_k)\ge 2 +2 tb(\gamma_{k_1+1})+ 2 tb(\gamma_{k_1+2})$. 
Assume $$ tb(\gamma_{k_1+1})\ge  \frac{2(s_{k_1+1}^2-1)}{3}-(k_1+1-2)s_{k_1+1} \textrm{ and } $$ 
$$ tb(\gamma_{k_1+2})\ge  \frac{2(s_{k_1+2}^2-1)}{3}-(k_1+2-2)s_{k_1+2}.$$
Then $$tb(\gamma_k))\ge 2 +\frac{4(s_{k_1+1}^2-1)}{3}-2(k_1-1)s_{k_1+1}+ \frac{4(s_{k_1+2}^2-1)}{3}-2k_1s_{k_1+2}.$$  
We show this last quantity is equal to  $$\frac{2}{3}(s_k^2-1)-(k-2)s_k.$$
\begin{itemize}
  \item  For $k_1=2^t$, $\lfloor\log_2(k_1-1)\rfloor =t-1$ and $\lfloor\log_2(2k_1-1)\rfloor = t$. 
  So we have $s_{k_1+1}=2^{t-1}$, $s_{k_1+2}=2^t$ and $s_{k}=2^t$.
  One can check that the two quantities are both equal to $-\frac{2}{3}+2^t-\frac{4}{3}2^{2t}$. \\
  \item  For $2^t <k_1<2^{t+1}$ we have $2^{t+1}-1 <2k_1-1<2^{t+2}-1$. 
    Then  $\lfloor\log_2k_1\rfloor=t$, $\lfloor\log_2(k_1-1)]\rfloor=t$ and $\lfloor\log_2(2k_1-1)\rfloor = t+1$. 
    So $s_{k_1+1}=2^t$, $s_{k_1+2}=2^t$ and $s_{k}=2^{t+1}$.
    One can check that the two quantities are both equal to $-\frac{2}{3}+\frac{8}{3}2^{2t}- (4k_1-2)2^{t}$. \\
    \end{itemize}

\end{proof}

%%%%%%%%%%%%%%%%%%%newsection%%%%%%%%%%%%%%%%%%%%%%%%%
\section{The graph Thurston-Bennequin number for complete bipartite graphs}\label{GTBB}

This section is concerned with the $TB$ of complete bipartite graphs $K_{n,m}$. 
For such a graph, denote by $P$ and $Q$ the subsets of vertices in the $n-$partition and $m-$partition, respectively.
Let  $w_{ae[P]}(K_{n,m})$ denote the total signed sum of crossings over all pairs of edges adjacent to a vertex in $P$. 
Let  $c_{v[P]}(K_{n,m})$ denote the total number of cusps at vertices in $P$, taken over all pairs of adjacent edges.

\begin{theorem} Let $K_{n,m}$ be a Legendrian embedding of a complete bipartite graph in $(\R^3, \xi_{std})$, with $n\ge m\ge 3$.  
Then $TB_{2r}(K_{n, m})$ is a multiple of $TB_4(K_{n, m})$. 
As a consequence, $TB(K_{n, m})$ is a multiple of $TB_4(K_{n, m})$. 
The quantities $TB_i(K_{n,m})$, $4\le i\le m$, can be computed from writhe and cusp counts of vertices and edges in the front projection rather than summing the $tbs$ of the cycles.
\label{Knm}
\end{theorem}

\begin{proof}
We consider the writhe and number of cusps in a Legendrian front projection for $K_{n, m}$. 
For the writhe, we consider crossings of an edge with itself, crossings between adjacent edges and crossings between non-adjacent edges.
To compute the number of cusps, we look at cusps along each edge and at cusps occurring at the vertices (between a pair of adjacent edges).
%In $K_{n, m}$, there are ${n \choose 2}{m \choose 2}$ cycles of length 4 and  ${n \choose r}{m \choose r } \frac{(r!)^2}{2r} = \frac{n! m!}{2r (n-r)! (m-r)!} $ cycles of length $2r$.\\

\begin{enumerate}
\item Each edge appears in $(n-1)(m-1)$ cycles of length 4 and  in $\frac{(n-1)!(m-1)!}{(n-r)!(m-r)!}$ cycles of length $2r$.
\item Each pair of adjacent edges appears in $(n-1)$ cycles of length 4 if the two edges are adjacent at a vertex in $P$ (the $n-$partition). 
 Each pair of adjacent edges appears in $(m-1)$ cycles of length 4 if the two edges are adjacent at a vertex in $Q$ (the $m-$partition). 
 Each pair of adjacent edges appears in $\frac{(m-2)!(n-1)!}{(m-r)!(n-r)!}$ cycles of length $2r$ if the two edges are adjacent at a vertex in $P$. 
 Each pair of adjacent edges appears in$\frac{(m-1)!(n-2)!}{(m-r)!(n-r)!}$ cycles of length $2r$ if the two edges are adjacent at a vertex in $Q$. 
 
 \item Each pair of non-adjacent edges appears in one cycle of length 4 and in \\$(2r-3)\frac{(m-2)!(n-2)!}{(m-r)!(n-r)!}$ cycles of length $2r$.
 We  obtain this count as follows: 
 The $r-2$ vertices in each partition that are in the cycle and not in one of the two non-adjacent edges can be chosen in ${n-2 \choose r-2}{m-2 \choose r-2} $ ways.  
 Without loss of generality choose one of the edges of interest to start constructing the cycle.  
 There are $2r-3$ positions where the second edge of interest can be placed in a $2r-$cycle, since it cannot be adjacent.  
 See the first row of in Figure \ref{Non-adjacent-edges}.  
 Each of the $r-2$ vertices in each partition can occupy $r-2$ remaining positions relative to the chosen two non-adjacent edges.  
 There are $(r-2)!(r-2)!$ ways to place these remaining vertices.  
 
 Fix an arbitrary embedding of $K_{n,m}$.  
 In $(r-1)\frac{(m-2)!(n-2)!}{(m-r)!(n-r)!}$ of these cycles, the non-adjacent edges intersect with one orientation, the same orientation as their intersection in the $4-$cycle. 
 In the other  $(r-2)\frac{(m-2)!(n-2)!}{(m-r)!(n-r)!}$ cycles, they intersect with the opposite orientation. 
 See the second row of Figure \ref{Non-adjacent-edges}.
 This means that the contribution of the two edges to the writhe comes from $\frac{(m-2)!(n-2)!}{(m-r)!(n-r)!}$ of the $2r-$cycles.
 For each of these cycles, the contribution is the same as the one from the $4-$cycle containing the two non-adjacent edges.

  %--Figure ---- count of cycles in Knm------
 
 \begin{figure}[htpb!]
\begin{center}
\begin{picture}(360, 80)
\put(0,0){\includegraphics[width=5.5in]{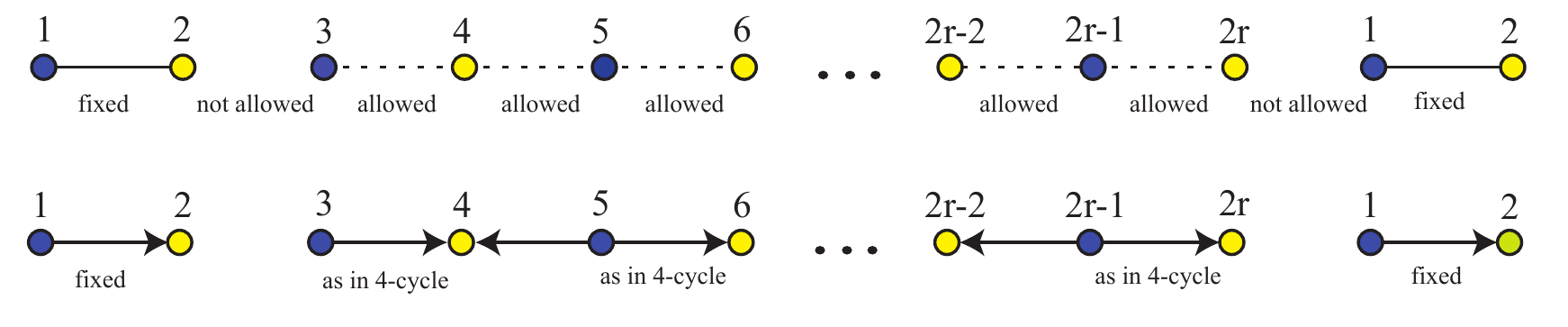}}
\end{picture}
\caption{\small Placement of two non-adjacent edges in a  $2r-$cycle.}
\label{Non-adjacent-edges}
\end{center}
\end{figure}

\end{enumerate}

Items (1)--(3) above give

$$TB_{2r}(K_{n,m})=TB_4(K_{n,m}) \cdot \frac{(m-2)!(n-2)!}{(m-r)!(n-r)!} .$$

Adding over all cycles gives
$$TB(K_{n,m})=   TB_4(K_{n,m})\cdot \ \sum_{r=2}^{m} \frac{(m-2)!(n-2)!}{(m-r)!(n-r)!} . $$
We can also compute $TB_4(K_{n,m})$ from the writhe and cusp count for edges and vertices as follows:\\\\
$TB_4(K_{n,m})=$
\[
\begin{array}{cl}
  & (n-1)(m-1)w_e(K_{n,m})+(n-1)w_{ae[P]}(K_{n,m})+(m-1)w_{ae[Q]}(K_{n,m})+w_{ne}(K_{n,m}) \\\\
  & -\frac{1}{2}[(n-1)(m-1)c_e(K_{n,m})+(n-1)c_{v[P]}(K_{n,m})+(m-1)c_{v[Q]}(K_{n,m})].
\end{array}
\]

\end{proof}

%---- Constructing a Knm with all 3-cycles trivial unknots ------

Here we describe a way to realize a minimal embedding of $K_{n,m}$.
See Figure \ref{fig-Knm}. 
Place all vertices of $K_{n,m}$ on the same horizontal line in the front projection, with the vertices in one partition first.
Then place the edges adjacent to a given vertex as nested arcs, with no intersections between adjacent edges.
See the left image of Figure \ref{fig-Knm} for such an embedding of $K_{5,3}.$
All $4-$cycles in this embedding have a front projection like that of the unknot on the right in the picture and therefore are trivial unknots.

%---- Figure ---------------------Knm-minimal ----------

\begin{figure}[htpb!]
\begin{center}
\begin{picture}(720, 90)
\put(0,0){\includegraphics[width=6.3in]{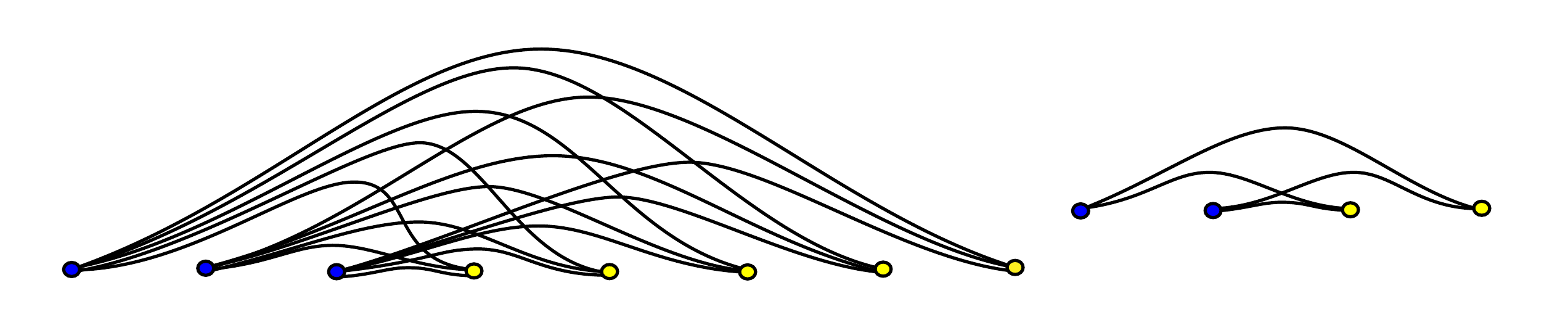}}
\end{picture}
\caption{\small Minimal embedding of $K_{5,3}$ and a $4-$cycle in this embedding.}
\label{fig-Knm}
\end{center}
\end{figure}

\begin{proposition}
For an unknotted minimal Legendrian $K_{3,3}$, three of its cycles of length six have $tb=-2$ and the other three have $tb=-1$.
Additionally, any pair of $6-$cycles with the same $tb$ share three non-adjacent edges, and any pair of $6-$cycles with different $tbs$ share two sets of adjacent edges.  
\end{proposition}
\label{prop_K33}
\begin{proof} For $n=m=3$ and $r=3$, Theorem \ref{Knm} says 
$$TB_{6}(K_{3,3})= TB_4(K_{3,3}).$$
The graph $K_{3,3}$ has nine cycles of length 4 and six cycles of length 6. If all nine $4-$cycles are of maximal $tb=-1$, then the sum of $tb$s over all six cycles is $-9$.
So there are at most three $6-$cycles with $tb=-2$.  
In the following we use our understanding of embeddings of $K_4$ to show that there are three $6-$cycles with $tb=-2$.
If we delete one of the edges of $K_{3,3}$, we obtain a subdivision of $K_4$, call it $K$.  
We will describe $K$ as a $K_4$ graph and ignore the valence 2 vertices, to simplify the explanation.    
 The $3-$cycles of $K$ are $4-$cycles of $K_{3,3}$. 
 See Figure \ref{K33sub4cyc}.
 
 %--Figure -------
 
 \begin{figure}[htpb!]
\begin{center}
\begin{picture}(280, 110)
\put(0,0){\includegraphics[width=3.5in]{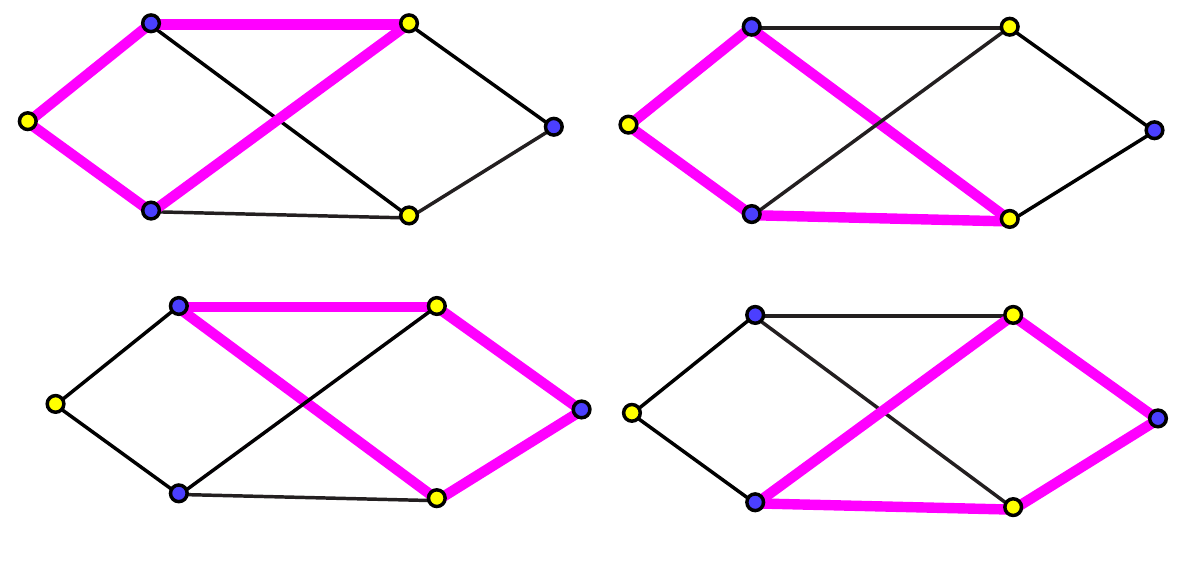}}
\end{picture}
\caption{\small The $3-$cycles of the $K_4$ subdivision correspond to $4-$cycles of $K_{3,3}$.}
\label{K33sub4cyc}
\end{center}
\end{figure}

\noindent One $4-$cycle of $K$ is a $4-$cycle of $K_{3,3}$, while the other two $4-$cycles of $K$ are $6-$cycles of $K_{3,3}$.
See Figure \ref{K33sub6cyc}.

 %--Figure -------
 
 \begin{figure}[htpb!]
\begin{center}
\begin{picture}(300, 70)
\put(0,0){\includegraphics[width=4.5in]{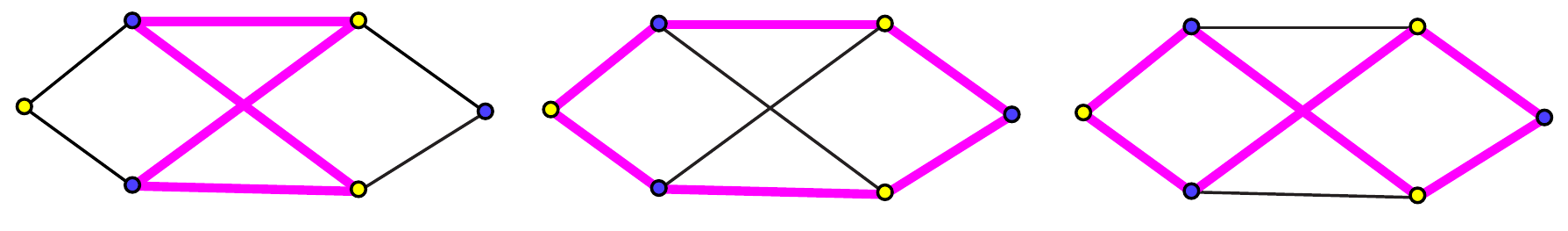}}
\end{picture}
\caption{\small The $4-$cycles of the $K_4$ subdivision correspond to a $4-$cycle of $K_{3,3}$ (left) and two $6-$cycles of $K_{3,3}$ (middle, right).}
\label{K33sub6cyc}
\end{center}
\end{figure}

By assumption, all $4-$cycles of $K_{3,3}$ have maximal $tb=-1$, so all $3-$cycles of  $K$ have maximal $tb=-1$. 
This means exactly one of the $4-$cycles of $K$ has $tb=-2$, with the other cycles having $tb=-1$.  
The $4-$cycles of $K$ are coming from a $4-$cycle of $K_{3,3}$ and two $6-$cycles of $K_{3,3}$.  
Since all $4-$cycles of $K_{3,3}$ have maximal $tb=-1$, one of the two $6-$cycles must have $tb=-2$ and the other must have $tb=-1$.  

All  edges of $K_{3,3}$ are equivalent up to graph automorphism.  
So all  $K_4$ subdivisions obtained by deleting a single edge of $K_{3,3}$  have the same structure.  
The set of three $K_4$ subdivisions shown in Figure \ref{K33subs} contain all of the $6-$cycles of $K_{3,3}$, with each cycle appearing once.  
So for each subdivision there is a different $6-$cycle with $tb=-2$.  
Thus there are exactly three $6-$cycles with $tb=-2$ and three $6-$cycles with $tb=-1$.

%---- Figure ---- K4 subdivision-----------

\begin{figure}[htpb!]
\begin{center}
\begin{picture}(300, 70)
\put(0,0){\includegraphics[width=4.5in]{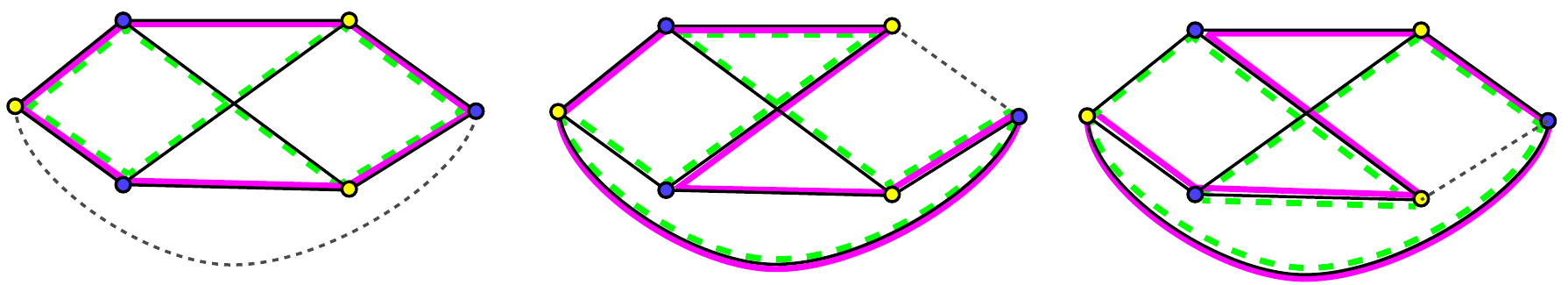}}
\end{picture}
\caption{\small By deleting the three edges adjacent to a vertex we see the six $6-$cycles of $K_{3,3}$ as $4-$cycles in the subdivision of $K_4$.}
\label{K33subs}
\end{center}
\end{figure}

Since all of the $K_4$ subdivisions obtained by deleting a single edge of $K_{3,3}$  have the same structure, any pair of $6-$cycles with different $tbs$ will have the same structure as those in Figure \ref{K33sub6cyc}.  
Thus, any two $6-$cycles with different $tbs$ will share two pairs of adjacent edges.  
Let $\gamma_6$ be an arbitrary $6-$cycle in an embedding of $K_{3,3}$.  
Consider the set of three $6-$cycles that share two pairs of adjacent edges with $\gamma_6$.  
See Figure \ref{K33proof}.
The cycles in this set will all have the same $tbs$.  
Any pair of these $6-$cycles with same $tbs$ share three non-adjacent edges.  

%-----Figure ----- K33proof------

\begin{figure}[htpb!]
\begin{center}
\begin{picture}(450, 80)
\put(0,0){\includegraphics[width=6in]{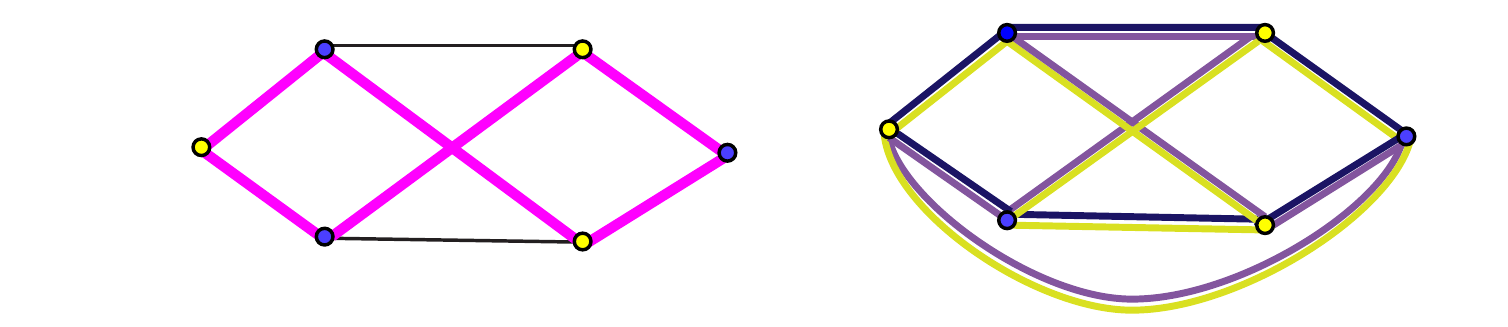}}
\end{picture}
\caption{\small The 6-cycle on the left shares two pairs of adjacent edges with each of three 6-cycles on the right. }
\label{K33proof}
\end{center}
\end{figure}

\end{proof}

From Proposition \ref{prop_K33}, we understand the structure of an unknotted minimal embedding of $K_{3,3}$ well.  
Up to relabelling the cycles, this gives one possible $tb(K_{3,3})$ for an unknotted minimal embedding.  
Since unknots with $tb=-2$ and $tb=-1$ are unique, there is a unique $tb$ and $|rot|$ for an unknotted minimal embedding of $K_{3,3}$.  
In Figure \ref{K33examples}, we give an unknotted minimal embedding of $K_{3,3}.$
In Section \ref{questions}, we show this embedding is equivalent to the one described after Theorem \ref{Knm}.

\begin{remark}
In a $K_{3,n}$ the vertices of a $6-$cycles define a unique $K_{3,3}$ subgraph.
Thus Proposition \ref{prop_K33} implies for unknotted minimal embeddings of $K_{3,n}$, half of the $6-$cycles have $tb=-2$ and half of them have $tb=-1$.
\end{remark}
%-------- Figure K33examples----------

\begin{figure}[htpb!]
\begin{center}
\begin{picture}(200, 70)
\put(0,0){\includegraphics[width=2.7in]{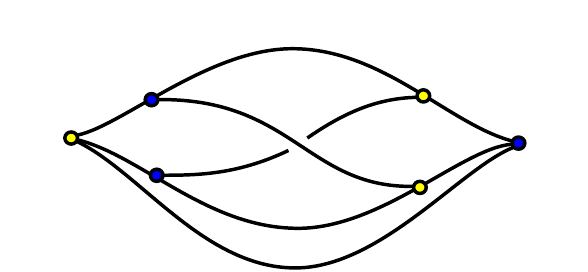}}
\end{picture}
\caption{\small An unknotted minimalembedding of $K_{3,3}$.}
\label{K33examples}
\end{center}
\end{figure}

Here we take a moment to consider complete bipartite graphs $K_{n,m}$, with $n\ge m$ and $m<3.$
There are no cycles in the complete bipartite graphs $K_{n,1},$ so they are of little interest.  
The complete bipartite graphs with $m=2$, i.e. $K_{n,2}$ are subdivisions of the $\theta_n-$graphs.  
For the $\theta_n-$graphs, the smallest cycles are $2-$cycles.  
These are the only cycles. 
So there cannot be a nice relationship between the cycles, like that seen earlier.  
For completeness we give a formula for the $TB$ of a $\theta_n-$graph in terms of writhe and cusp counts of vertices and edges in the front projection.  
We consider the writhe and number of cusps in a diagram for a $\theta_n-$graph, call the diagram $\theta_n$. 
%For the writhe, we consider crossings of an edge with itself, and crossings between adjacent edges.
In a $\theta_n-$graph there are no non-adjacent edges.
%To compute the number of cusps, we look at cusps along each edge and at cusps occurring at the vertices, between a pair of adjacent edges.
%In a $\theta_n-$graph there are ${n \choose 2}$ cycles.
Each edge appears in $(n-1)$ cycles and each pair of adjacent edges makes up one of the cycles.
This gives 
$$TB(\theta_n)=(n-1)w_e(\theta_n)+w_{ae}(\theta_n)-\frac{1}{2}[(n-1)c_e(\theta_n)+c_v(\theta_n)].$$

%---Section - -- questions

\section{Questions and examples of embeddings}\label{questions}
In this section we consider minimal embeddings of $K_4$, $K_5$ and $K_{3,3}$.  
We show the equivalence of diagrams of unknotted minimal embeddings discussed earlier.  
First we will recall the Reidemeister moves for Legendrian graphs.  
Two generic front projections of a Legendrian graph are related by Reidemeister moves I, II and III together with three moves given by the mutual position of vertices and edges \cite{BI}. See Figure~\ref{moves}.

%---- Figure ---- Reidemeister moves + + + -----------

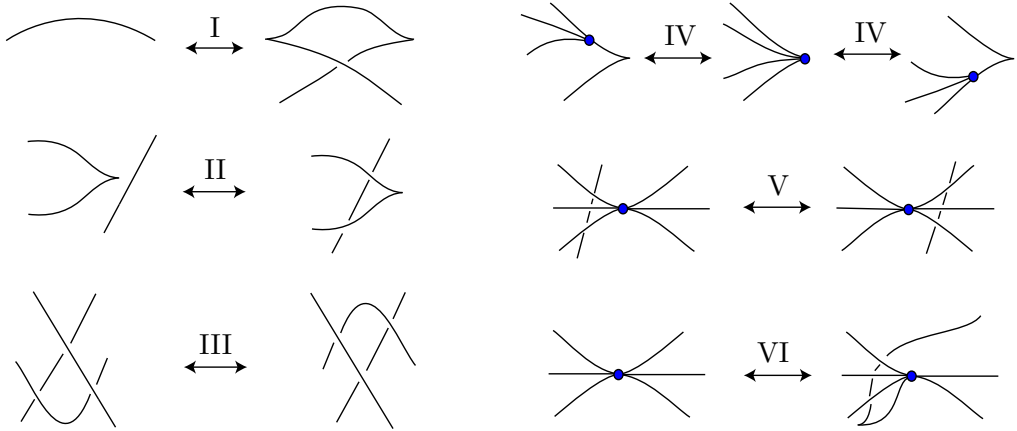
\begin{figure}[htpb!]
\begin{center}
\begin{picture}(389, 170)
\put(0,0){\includegraphics[width=5.5in]{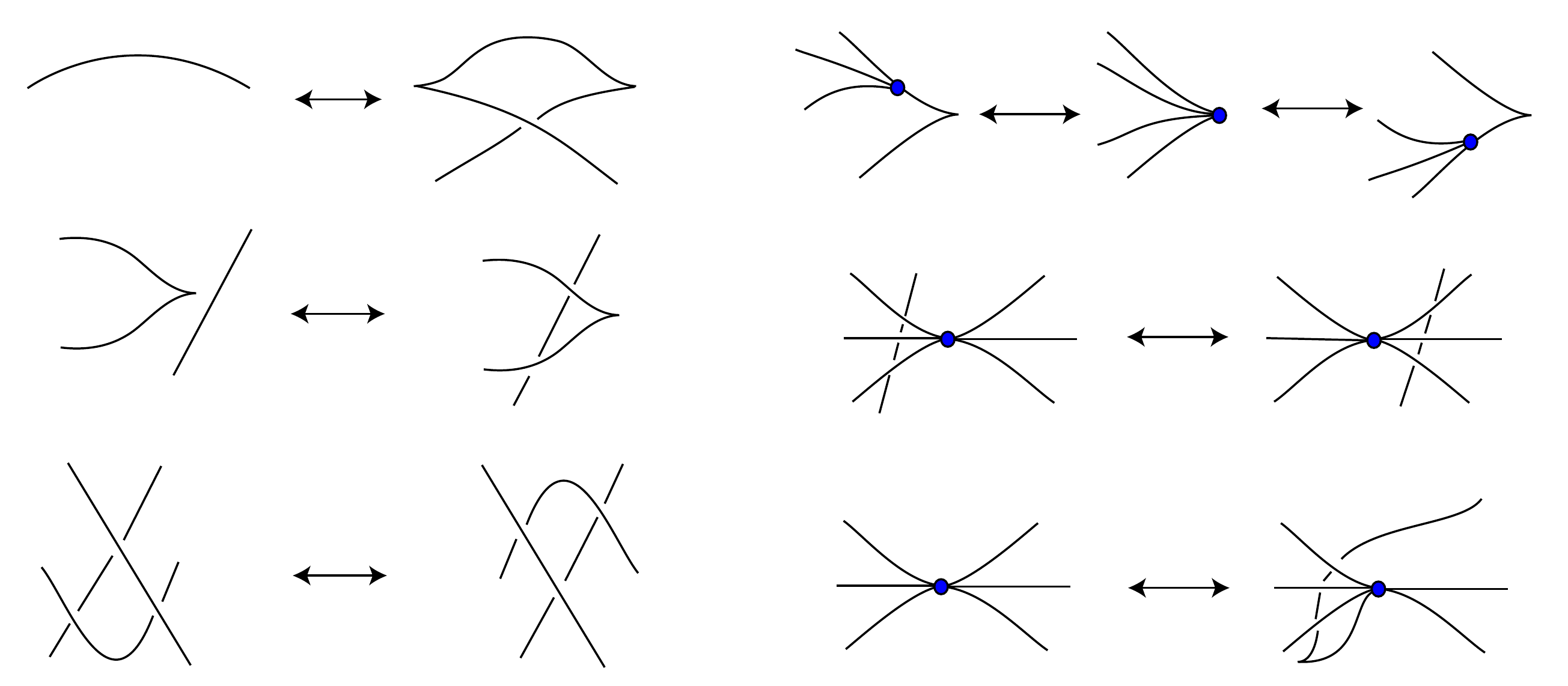}}
\put(84, 153){I}
\put(82,99){II}
\put(80,33){III}
\put(256, 150){IV}
\put(328, 151){IV}
\put(295,93){V}
\put(291,30){VI}
\end{picture}
\caption{\small Legendrian isotopy moves for graphs:  Reidemeister moves I, II, and III, a vertex passing through a cusp (IV), an edge passing under or over a vertex (V), an edge adjacent to a vertex rotates to the other side of the vertex (VI). Reflections of these moves that are Legendrian front projections are also allowed.}\label{moves}
\end{center}
\end{figure}

In Figure \ref{K4examples}, we show two diagrams for the one unknotted minimal embedding of $K_4$ that is known.  
In Figure \ref{K4flat}, we show four diagrams of unknotted minimal embeddings of $K_4$ without crossings.  
%---------Figure -----
\begin{figure}[htpb!]
\begin{center}
\includegraphics[width=3.5in]{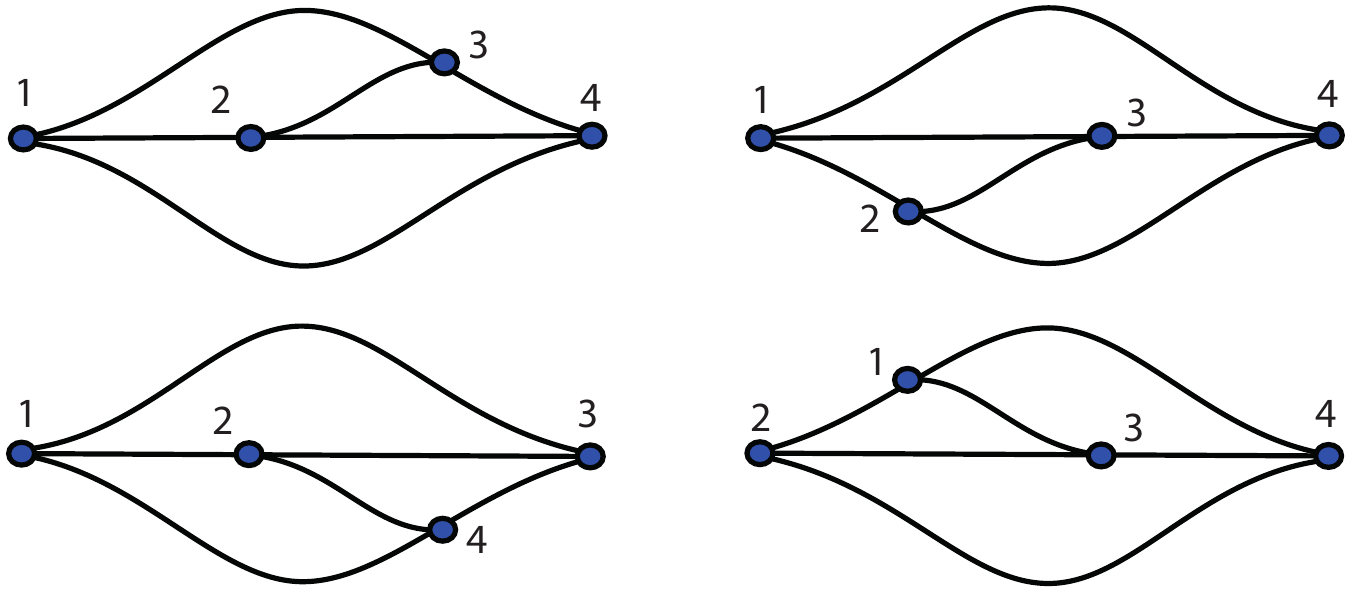}
\caption{\small Four unknotted minimal embeddings of $K_4$.}
\label{K4flat}
\end{center}
\end{figure}
To go between the top and bottom diagrams, in each column of Figure \ref{K4flat}, it takes two Reidemeister IV moves (the vertices are number to make it easier to see how this is done).  
In Figure \ref{K4flat-flat}, we show how to go between the two diagrams in the top row of Figure \ref{K4flat}.
Thus all diagrams in Figure \ref{K4flat} are equivalent.
%----------- Figure ---------
\begin{figure}[htpb!]
\begin{center}
\includegraphics[width=6in]{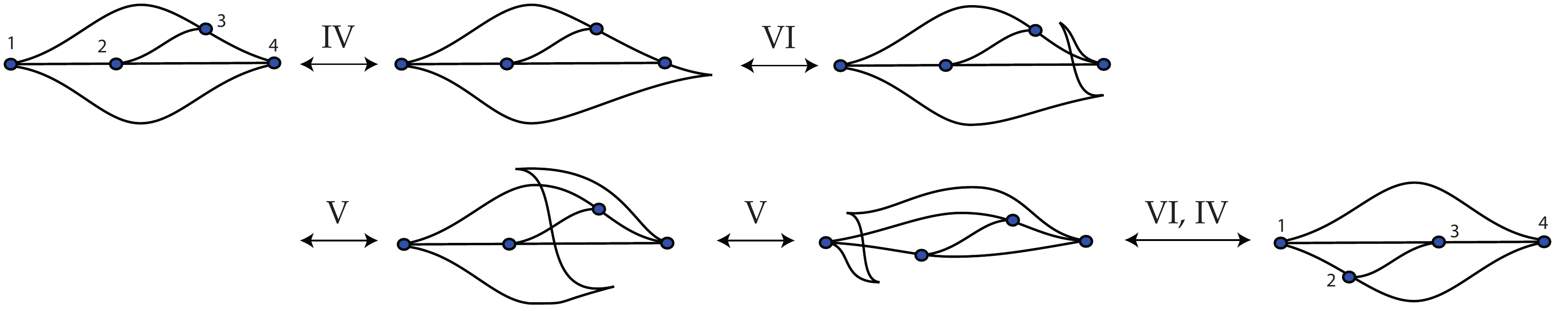}
\caption{\small A sequence of Legendrian $K_4$s related by Reidemeister moves. }
\label{K4flat-flat}
\end{center}
\end{figure}

\noindent Finally in Figure \ref{K4other-flat}, we give the more complicated sequence showing that the left diagram from Figure \ref{K4examples} is equivalent with the final diagram in Figure \ref{K4flat-flat}.   
Thus the diagrams in Figure \ref{K4examples} are equivalent.

%----------- Figure ---------
\begin{figure}[htpb!]
\begin{center}
\includegraphics[width=6in]{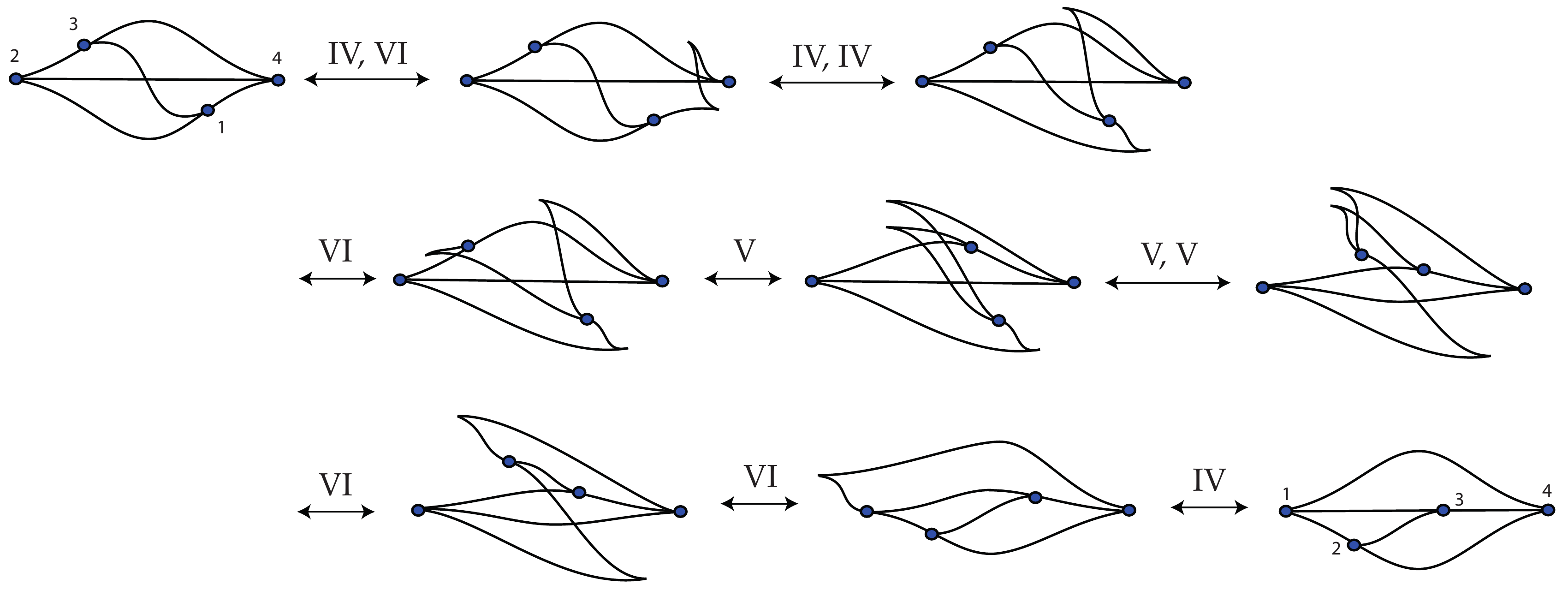}
\caption{\small A sequence of Legendrian $K_4$s related by Reidemeister moves. }
\label{K4other-flat}
\end{center}
\end{figure}
%All of the minimal diagrams for $K_4$ that we have found we have determined to be equivalent, this leads us to the following question:
We do not know of a different unknotted minimal embedding of $K_4$, which leads us to the following question:
\begin{question}
Is this the unique unknotted minimal embedding of $K_4$?
\end{question}

\noindent If we consider minimal embeddings rather than unknotted minimal embeddings there are a number of other possibilities.
In Figure \ref{K4-k-example}, we give an infinite family of examples.  
% of embeddings of $K_4$ with $k$ odd, where:
For each $k\in\Z$ odd, this is a minimal embedding of $K_4$ where:
\begin{itemize}
%\item  all of the $3-$cycles have $tb=-1$,
\item one $4-$cycle is an unknot with $tb=-1$,
\item one $4-$cycle is an unknot with $tb=-k-1$ ($rot=\pm 1$), and
\item one $4-$cycle is a $(2,k)-$torus knot with $tb=k-2$ ($rot=0$).
\end{itemize}  

%----------- Figure ---------
\begin{figure}[htpb!]
\begin{center}
\includegraphics[width=3in]{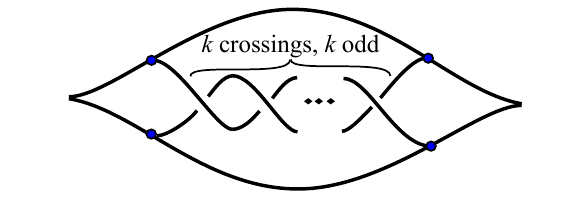}
\caption{\small Legendrian $K_4$ with one $4-$cycle a $(2,k)-$torus knot with $tb=k-2$. }
\label{K4-k-example}
\end{center}
\end{figure}

\begin{question}
Is there an embedding of $K_4(L)$ where $tb_{max}(L)=0$ and all cycles have maximal $tb$?
\end{question}

For $K_5,$ we have shown three different unknotted minimal embeddings, two with the sequence $s_2$ for the $tbs$ of the $5-$cycles and one with $s_3$ for the $tbs$ of the $5-$cycles.  
However, based on our calculations there are many other possible minimal embeddings.  
\begin{question}
%Are there minimal embeddings of $K_5$ realizing the sequence $s_3$ for the $tbs$ for the $5-$cycles with one or both of the rotation numbers of the cycles with $tb=-3$ being $\pm 2$?
Are there unknotted minimal embeddings of $K_5$ realizing the sequence $s_3$  different than the one in Figure \ref{K5examples}? (One way this could occur is if one or both $5-$cycles with $tb=-3$ had rotation number $\pm 2$.)

\end{question}

\begin{question}
Are there unknotted minimal embeddings of $K_5$ realizing any of the other sequences of $tbs$ for the $5-$cycles $(s_1, s_4, s_5, s_6, s_7, s_8, s_9, s_{10})$?
\end{question}

\begin{conjecture}
Any unknotted minimal embedding of $K_5$ will contain at least one $5-$cycle with $tb=-3$, and will not contain a $5-$cycle with $tb=-4$.  
\end{conjecture}
\noindent Building on our example of minimal embeddings for $K_4$ in Figure \ref{K4-k-example}, there are also infinitely many different possible minimal embeddings of $K_5$.  

In Figure \ref{K33g-g}, we show that the $K_{3,3}$ embedding described after Theorem \ref{Knm} is the same as that shown in Figure \ref{K33examples}.
For the graph $K_{3,3}$, having an embedding with all its smallest cycles trivial unknots seems to be a more rigid constraint than it is for complete graphs.  
%If we remove the restriction that the $6-$cycles are also unknotted, we have not been able to find other embeddings of $K_{3,3}$.  
We have not found any other unknotted minimal embeddings or minimal embedding of $K_{3,3}$.  
%----------- Figure ---------
\begin{figure}[htpb!]
\begin{center}
\includegraphics[width=6in]{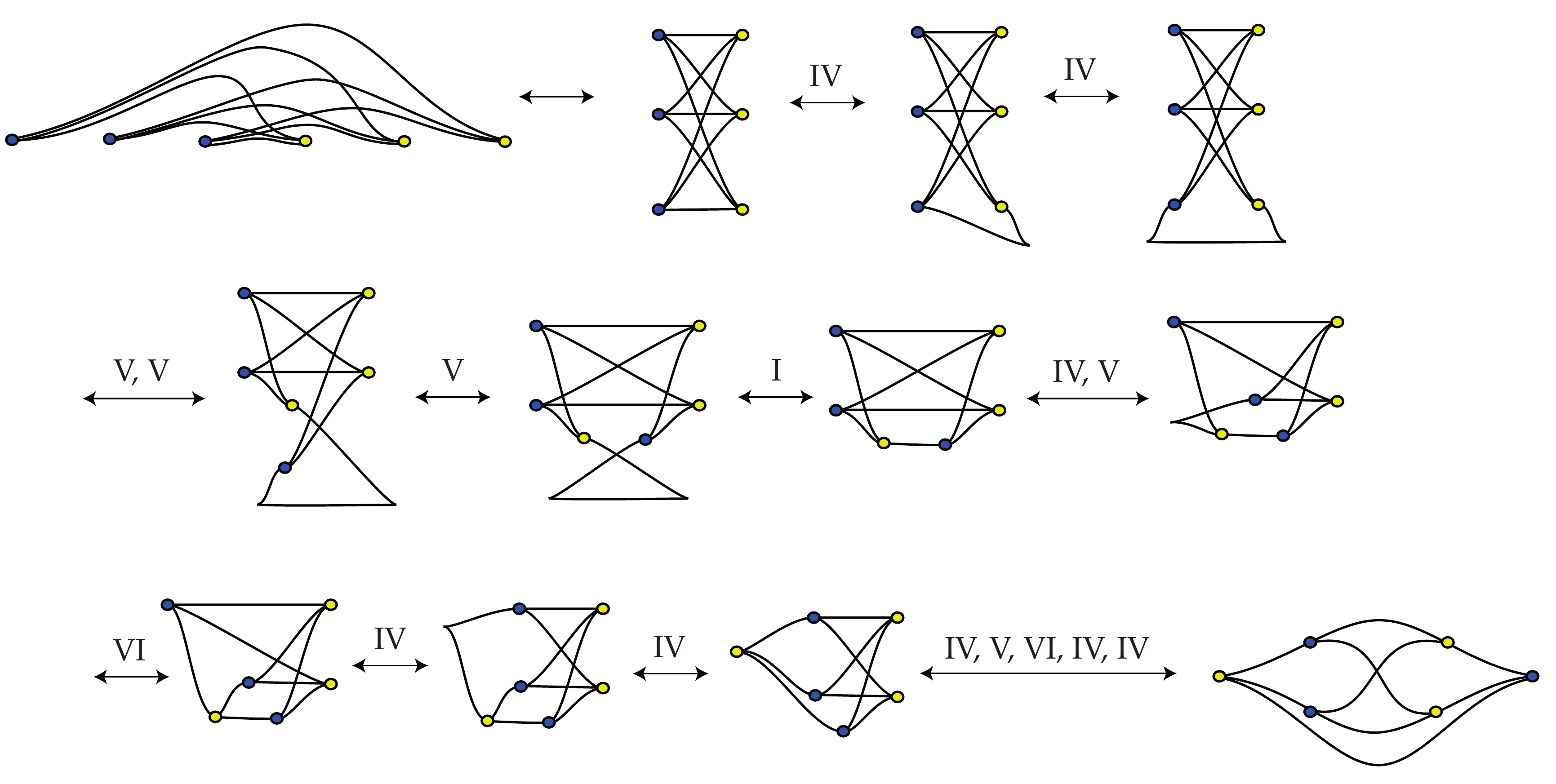}
\caption{\small A sequence of Legendrian $K_{3,3}$s related by Reidemeister moves and planar isotopy. }
\label{K33g-g}
\end{center}
\end{figure}

\begin{question}
Is the embedding shown in Figure \ref{K33examples} the unique unknotted minimal embedding of $K_{3,3}$?
\end{question}

\begin{question}
Is the embedding shown in Figure \ref{K33examples} the unique minimal embedding of $K_{3,3}$?
\end{question}

%---Bibliography---------------
\bibliographystyle{amsplain}

\end{document}